\documentclass[10pt, letterpaper]{article}

\usepackage{amsmath, amssymb, amsthm, graphicx, amscd, amsfonts, mathrsfs, mathtools}
\usepackage{geometry}
\usepackage{hyperref}
\usepackage{setspace}
\usepackage{microtype} 
\usepackage{enumitem}

\newtheorem{theorem}{Theorem}[section]
\newtheorem{lemma}[theorem]{Lemma}
\newtheorem{proposition}[theorem]{Proposition}
\newtheorem{corollary}[theorem]{Corollary}
\theoremstyle{definition}
\newtheorem{definition}[theorem]{Definition}

\newcommand{\F}{\mathbb{F}}

\newcommand{\bb}[1]{\mathbb{#1}}
\newcommand{\bd}{\mathbf{d}}
\newcommand{\Tr}{\operatorname{Tr}}
\newcommand{\ord}{\operatorname{ord}}
\newcommand{\lcm}{\operatorname{lcm}}
\newcommand{\supp}{\operatorname{supp}}
\newcommand{\Gal}{\operatorname{Gal}}
\newcommand{\Teich}{\operatorname{Teich}}
\newcommand{\Frob}{\operatorname{Frob}}

\begin{document}

\title{$p$-adic Theory for Partial Toric Exponential Sums}
\author{C. Douglas Haessig}
\date{\today}

\maketitle

\begin{abstract}
Wan proved the rationality of partial toric $L$-functions using $\ell$-adic techniques \cite{MR2027782}. In this paper, we present a $p$-adic proof in the spirit of Dwork. We demonstrate that partial $L$-functions can be expressed as an alternating product of twisted Fredholm determinants. These twisted determinants appear to be intrinsic to the analytic structure of partial $L$-functions, and unlike their classical counterparts, twisted Fredholm determinants of completely continuous operators are not automatically $p$-adic entire functions. However, for partial $L$-functions they will be $p$-adic meromorphic.

After proving rationality, we construct a $p$-adic cohomology theory and give a $p$-adic cohomological formula for partial toric $L$-functions. Last, we show they have a unique $p$-adic unit root which may be explicitly written in terms of $\mathcal{A}$-hypergeometric series.
\end{abstract}

\tableofcontents
\vspace{1em}

\section{Introduction}
\label{sec:intro}

Let $\bb{F}_q$ be the finite field of $q = p^a$ elements of characteristic $p$, and let $f \in \bb{F}_q[x_1^{\pm}, \ldots, x_n^{\pm}]$ be a Laurent polynomial in $n$ variables. Fix a nontrivial additive character $\psi: \bb{F}_q \to \bb{C}_p^\times$. In the classical theory of toric exponential sums, one studies $\sum \psi \circ  \Tr_{\bb{F}_{q^k} / \bb{F}_q}(f(x)) $ as the sum runs over $x \in (\bb{F}_{q^k}^*)^n$.

In this paper, we are interested in an asymmetric generalization known as \emph{partial toric exponential sums}, where the variables lie in possibly distinct extensions of $\bb F_q$. Let $\bd = (d_1, \dots, d_n) \in \bb{Z}_{\ge 1}^n$ and set $d := \lcm(d_1, \dots, d_n)$. For each integer $k \ge 1$, we define the \emph{partial toric exponential sum} of $f$ with respect to $\bd$ as
\begin{equation}\label{eq:def_Sk}
S_k(f, \bd) := \sum_{x_1, \dots, x_n} \psi \circ \Tr_{\bb{F}_{q^{kd}} / \bb{F}_q} \big( f(x_1, \dots, x_n) \big),
\end{equation}
where the summation runs over all $x_i \in \bb{F}_{q^{k d_i}}^*$ for $1 \le i \le n$. Define the associated partial $L$-function by
\begin{equation}\label{eq:def_L}
L(\bd, f/\bb{F}_q, T) := \exp \left( \sum_{k=1}^\infty S_k(f, \bd) \frac{T^k}{k} \right) \in 1 + T\bb{Q}(\zeta_p)[[T]].
\end{equation}

Dwork's classical rationality proof \cite{Dwork-rationalityofzeta-1960} follows from three properties of the $L$-function: $p$-adic meromorphic, complex analytic, and the coefficients lie in a finite extension of $\bb Q$. The partial $L$-function satisfies the last two properties in the same way classical $L$-functions of toric exponential sums satisfy them. The obstruction here lies in the $p$-adic meromorphy. 

Wan \cite{MR2027782} proved the rationality of such partial $L$-functions and partial zeta functions using $\ell$-adic cohomological methods. Strongly guided by his proof, we provide a $p$-adic proof of this rationality. The main obstacle in constructing the theory is that everything is twisted by a permutation $\sigma$. This means we need to prove a twisted version of Dwork's trace formula, and the Fredholm determinants are twisted by $\sigma$. A twisted Fredholm determinant is defined by 
\[
\det\nolimits_{\mathcal{H}}(I - T\Phi \mid V) := \exp \left( - \sum_{k=1}^\infty \Tr(\mathcal{H} \circ \Phi^k \mid V) \frac{T^k}{k} \right).
\]
Even for completely continuous operators, it may not be the case that twisted Fredholm determinants are $p$-adic entire. However, in the case of partial $L$-functions, we show they the corresponding twisted Fredholm determinants are $p$-adic meromorphic. This gives us all we need to prove rationality.

\begin{theorem}
The partial $L$-function $L(\bd, f/\bb{F}_q, T)$ is a rational function in $\bb{Q}(\zeta_p)(T)$.
\end{theorem}

We also construct a $p$-adic cohomology theory for partial $L$-functions. The underlying complex and differentials are the usual ones coming from Adolphson and Sperber's theory, but the Frobenius chain map must be twisted by $\sigma$ (which corresponds to $\mathcal{H}^{(m)}$ in the following):

\begin{theorem}\label{thm:main_rationality}
There exists a $p$-adic cohomology theory $H^\bullet(B, D)$ such that
\[
L(\bd, f/\bb{F}_q, T) = \prod_{m=0}^N \det\nolimits_{\mathcal{H}^{(m)}} \! \left(I - T\Frob_q^m \ \Big|\ H^m(B, D)\right)^{(-1)^{m+1}},
\]
where $N = \sum_{i=1}^n d_i$.
\end{theorem}

Even under non-degeneracy conditions, where the cohomology is acyclic and thus the $L$-function takes the form $L(\bd, f/\bb{F}_q, T)^{(-1)^{N+1}} = \det\nolimits_{\mathcal{H}^{(N)}} \! \left(I - T\Frob_q^N \ \Big|\ H^N(B, D)\right)$, this may not be a polynomial since the twisted Fredholm determinant is only known to be $p$-adic meromorphic. This leaves open what form a ``Newton over Hodge'' statement should be in this context.

We note that the ``twisted'' natural of these $L$-functions comes from the unfolding technique in Section \ref{S: unfolding}. We believe this technique is intrinsic to these partial exponential sums based on an earlier investigation \cite{MR4531526}, not an artifact of method. 

Now that we have a $p$-adic theory for partial $L$-functions, we may apply established $p$-adic tools in their study, although a slight modification will likely be needed due to the twist. To demonstrate, in Sections \ref{sec:unique_root} and \ref{sec:unit_root_formula} we prove the partial $L$-function has a unique $p$-adic unit root and use Adolphson and Sperber's theory to give a formula for the unit root in terms of $\mathcal{A}$-hypergeometric series:

\begin{theorem}\label{thm:main_unit_root}
The partial $L$-function $L(\bd, f/\bb{F}_q, T)^{(-1)^{n-1}}$ has a unique $p$-adic unit root $\lambda_0$ which is a 1-unit. Furthermore, there exists a $p$-adically analytic function $\mathcal{F}(\Lambda)$ which is a ratio of $\mathcal{A}$-hypergeometric series such that:
\[
\lambda_0 = \mathcal{F}(\hat c) \mathcal{F}(\hat c^p) \cdots \mathcal{F}(\hat c^{p^{a - 1}}).
\]
(See Theorem \ref{thm:main_unit_root_formula} for a precise statement.)
\end{theorem}

Partial $L$-functions and the related partial zeta functions are rather new, and so only a few papers have appeared in their study. Wan introduced them in \cite{MR1803946}, and proved rationality in \cite{MR2027782}. Fu and Wan continued their $\ell$-adic study in \cite{MR1960121, MR2030375}. The partial zeta function of Fermat hypersurfaces was studied in \cite{MR4318328}, and some $p$-adic results are given in \cite{MR4531526}.

\section{The unfolding polynomial $G$}\label{S: unfolding}

Let $\bb{F}_q$ be the finite field of $q = p^a$ elements of characteristic $p$, and let $f \in \bb{F}_q[x_1^{\pm}, \ldots, x_n^{\pm}]$ be a Laurent polynomial in $n$ variables. Set $N := \sum_{i=1}^n d_i$. We think of the variables $x_i$ as being ``folded'' inward into smaller subfields. To unfold them, guided by Wan \cite{MR2027782}, we introduce $N$ independent variables $y_{i, j}$, where $1 \le i \le n$ and $j \in \bb Z/d_i\bb Z$. For $v \in \bb Z^N$, we will use the multi-index notation $y^v = \prod y_{i, j}^{v_{i, j}}$ where the product runs over all $1 \le i \le n$ and $j \in \bb Z/ d_i \bb Z$. Define the shift operator $\sigma$ on the variables $y_{i, j}$ by
\[
\sigma(y_{i, j \bmod d_i}) := y_{i,(j+1) \bmod d_i}.
\]
For functions $\xi(y)$ in the variables $y$, define the action of $\sigma$ by $(\sigma \xi)(y) := \xi(\sigma y)$. It will be useful to have a matrix description of this action. Denote by $P$ the $N \times N$ permutation matrix acting on $v \in \bb Z^N$ by $(Pv)_{i,j} := v_{i, (j-1)\bmod d_i}$.

\begin{lemma} \label{lem:pullback_action}
$\sigma(y^v) = y^{Pv}$, and $\sigma$ has order $d = \lcm(d_1, \dots, d_n)$.
\end{lemma}

\begin{proof}
While an easy exercise to prove, we give it here since it is used often. By definition, on a monomial $y^v$ we have
\[
\sigma(y^v) = \sigma \prod_{i=1}^n \prod_{j=0}^{d_i-1} y_{i,j}^{v_{i,j}} = \prod_{i=1}^n \prod_{j=0}^{d_i-1} (\sigma y_{i,j})^{v_{i,j}}  = \prod_{i=1}^n \prod_{j=0}^{d_i-1} y_{i, (j+1)\bmod d_i}^{v_{i,j}}.
\]
Re-indexing the inside product by setting $k = (j+1) \bmod d_i$, we obtain:
\[
\prod_{i=1}^n \prod_{k=0}^{d_i-1} y_{i, k}^{v_{i, (k-1)\bmod d_i}} = \prod_{i=1}^n \prod_{k=0}^{d_i-1} y_{i, k}^{(Pv)_{i, k}} = y^{Pv}.
\]
That $\sigma$ has order $d$ is immediate from $d = \lcm(d_1, \dots, d_n)$.
\end{proof}

Define the \emph{unfolded} Laurent polynomial $G(y)$ associated to $f$ by 
\[
G(y) := \sum_{l=0}^{d-1} f\big(y_{1, l \bmod d_1}, \dots, y_{n, l \bmod d_n}\big) \in \bb F_q[y_{i,j}^{\pm 1} \mid 1 \le i \le n, j \in \bb Z/ d_i \bb Z].
\]

\begin{lemma}\label{lem:G_shift_invariant}
The unfolded polynomial $G$ is invariant under $\sigma$: $\sigma G(y) = G(y)$.
\end{lemma}

\begin{proof}
Calculating,
\begin{align*}
G(\sigma y) &= \sum_{l=0}^{d-1} f\Big((\sigma y)_{1, l \bmod d_1}, \dots, (\sigma y)_{n, l \bmod d_n}\Big) \\
&= \sum_{l=0}^{d-1} f\Big(y_{1, (l+1) \bmod d_1}, \dots, y_{n, (l+1) \bmod d_n}\Big).
\end{align*}
Since $d$ is a multiple of every $d_i$, the total sum remains unchanged if we shift the index by setting $l' := l+1$. Thus, $G(\sigma y) = G(y)$.
\end{proof}

Let $\Delta := \Delta(G) \subseteq \bb R^N$ be the Newton polytope of $G(y)$ at infinity, which is defined as the convex hull in $\bb R^N$ of the origin and the vectors in the support of $G$. Define the monoid $M := \bb R_{\ge 0} \Delta \cap \bb Z^N$. Let $w: M \to \bb Q_{\ge 0}$ be the associated polyhedral weight function defined as follows. For $u \in M(\Delta)$, $w(u)$ is the smallest dilation $\delta$ such that $u \in \delta \Delta$:
\[
w(u) := \inf \{ \delta \mid u \in \delta \Delta \}.
\]
The weight function satisfies: 
\begin{enumerate}
\item $w(M) = \frac{1}{D} \bb Z_{\geq 0}$ for some positive integer $D$.
\item $w(u) = 0$ if and only if $u = 0$. 
\item For $c \in \bb Q_{\ge 0}$, $w(cu) = cw(u)$.
\item $w(u+v) \le w(u) + w(v)$ with equality if and only if $u$ and $v$ are cofacial on $\Delta$.
\end{enumerate}

The unfolded polynomial $G$ is highly symmetric, and so $\Delta$ is highly symmetric. As a result, the weight function is invariant by the action of the permutation $P$ associated to $\sigma$:

\begin{lemma}\label{lem:weight_invariant}
For all $u \in M$, $w(Pu) = w(u)$.
\end{lemma}

\begin{proof}
By Lemma \ref{lem:G_shift_invariant}, $G(\sigma y) = G(y)$ and so the support of $G$ is invariant under the action by $P$. Consequently, the Newton polytope $\Delta(G) = \text{Convex}(\{0\} \cup \text{supp}(G))$ satisfies $P(\Delta) = \Delta$. For any $u \in M$ and $\delta \ge 0$,
\[
u \in \delta\Delta \iff Pu \in P(\delta \Delta) = \delta \Delta(G),
\]
and so $w(Pu) = w(u)$.
\end{proof}

\section{The $p$-adic Banach Space and Dwork's Frobenius}\label{S: Banach}

Let $\bb Q_q$ be the unramified extension of $\bb Q_p$ of degree $a$, and let $\bb Z_q$ be its ring of integers. By the theory of Newton polygons, the series $\sum_{i \ge 0} t^{p^i}/p^i$ has a root $\gamma \in \overline{\bb Q}_p$ satisfying $\ord_p \gamma = 1/(p-1)$. Fix a $D$-th root $\gamma^{1/D}$ of $\gamma$, and set $K := \bb Q_q(\gamma^{1/D})$. Define the $p$-adic Banach space $B$ over $K$ with orthonormal basis $\{ \gamma^{w(u)} y^u \}_{u \in M}$ as:
\[
B := \left\{ \sum_{u \in M} A_u \gamma^{w(u)} y^u \;\middle|\; A_u \in K \text{ and }  |A_u|_p \rightarrow 0 \text{ as } w(u) \to \infty \right\}.
\]

Let $E(t) = \exp(\sum_{i \ge 0} t^{p^i}/p^i)$ be the Artin-Hasse exponential, and define Dwork's splitting function $\theta(t) := E(\gamma t)$. It is well-known that $\theta(t) = \sum_{i=0}^{\infty}\theta_it^i$ has coefficients which satisfy $\mathrm{ord}_p(\theta_i) \geq \frac{i}{p-1}$. 

Writing $f(x) = \sum_{u \in \supp(f)} \bar c_u x^u$ with coefficients $\bar c_u \in \bb F_q^*$, let $c_u \in \bb Z_q$ be their Teichm\"uller lifts. Dwork's splitting function is a $p$-adic analytic lifting of the character $\zeta_p^{\Tr_{\bb F_q / \bb F_p}( \cdot)}$. That is, for $\bar z \in \bb F_q^*$ and $\hat z = \Teich(\bar z) \in \bb Z_q$ its Teichm\"uller lift,
\[
\zeta_p^{\Tr_{\bb F_q / \bb F_p}(\bar z)} = \zeta_p^{\bar z + \bar z^p + \cdots + \bar z^{p^{a-1}} } = \theta(\hat z) \theta(\hat z^p) \cdots \theta(\hat z^{p^{a-1}}).
\]
By definition, $G(y) = \sum_{l = 0}^{d-1} \sum_{u \in \supp(f)} \bar c_u \prod_{i=1}^n y_{i, l \bmod d_i}^{u_i}$. Define
\[
F(y) := \prod_{l=0}^{d-1} \prod_{u \in \supp(f)} \theta\!\left( c_u \prod_{i=1}^n y_{i, l \bmod d_i}^{u_i} \right).
\]
Let $\tau \in \Gal(\bb Q_q / \bb Q_p)$ be the generator of the cyclic Galois group which acts on Teichm\"ullers by $\tau(c) = c^p$. Define
\begin{align*}
F_a(y) &:= F(y) F^\tau(y^p) \cdots F^{\tau^{a-1}}(y^{p^{a-1}}) \\
&= \prod_{m=0}^{a-1} \prod_{l=0}^{d-1} \prod_{u \in \supp(f)} \theta\!\left( c_u^{\,p^m} \prod_{i=1}^n y_{i, l \bmod d_i}^{p^m u_i} \right).
\end{align*}

The same argument as in Lemma \ref{lem:G_shift_invariant} proves that $F$ and $F_a$ are invariant under $\sigma$:

\begin{lemma}\label{L: F invariant}
$\sigma F(y) = F(y)$ and $\sigma F_a(y) = F_a(y)$
\end{lemma}

Define the operator $\psi_p$ on $B$ by its action on monomials: $\psi_p(y^u) = y^{u/p}$ if $p$ divides every coordinate of $u$, and $0$ otherwise. We define Dwork's Frobenius operator on $B$ by
\[
\alpha := \tau^{-1} \circ \psi_p \circ F
\]
where $F$ is the map defined via multiplication by the series $F(y)$. Define $\psi_q := \psi_p^a$ and $\alpha_a := \psi_q \circ F_a(y)$.

In order to show that $\alpha$ and $\alpha_a$ are well-defined operators on $B$, we need the following estimate. Write $F(y) = \sum_{u \in M} B_u y^u$. Set $\ord_\gamma := (p-1) \ord_p$.

\begin{lemma}\label{L: F estimates}
For all $u \in M$, $\ord_\gamma (B_u) \ge w(u)$.
\end{lemma}

\begin{proof}
To ease notation, let the terms of the unfolded polynomial $G(y)$ be indexed by $m = 1, \dots, k$. That is, write $G(y) = \sum_{m=1}^k \bar c_m y^{v^{(m)}}$. Then,
\begin{align*}
F(y) &= \prod_{m=1}^k \theta(c_m y^{v^{(m)}}) \\
&= \sum_{i_1,\dots,i_k \ge 0} \theta_{i_1}\dots\theta_{i_k} c_1^{i_1}\dots c_k^{i_k} y^{i_1 v^{(1)} + \dots + i_k v^{(k)}} \\
&= \sum_{u \in M} B_u y^u
\end{align*}
where
\[
B_u := \sum_{i_1 v^{(1)} + \dots + i_k v^{(k)} = u} \theta_{i_1}\dots\theta_{i_k} c_1^{i_1}\dots c_k^{i_k}.
\]
By the properties of the weight function, and since $\ord_\gamma(\theta_i) \ge i$, we have
\begin{align*}
\ord_\gamma(\theta_{i_1}\dots\theta_{i_k} c_1^{i_1}\dots c_k^{i_k} ) &\ge i_1 w(v^{(1)}) + \dots + i_k w(v^{(k)}) \\
&\geq w(i_1 v^{(1)} + \dots + i_k v^{(k)}) \\
&= w(u).
\end{align*}
The lemma follows.
\end{proof}

We may now show that $\alpha$ and $\alpha_a$ are well-defined operators of $B$.

\begin{proposition}
The operator $\alpha := \tau^{-1} \circ \psi_p \circ F$ is a well-defined operator on the Banach space $B$.
\end{proposition}

\begin{proof}
Let $\xi(y) = \sum_{u \in M} C_u y^u \in B$. By the definition of $B$, $\ord_\gamma(C_u) \ge w(u) + c_u$, where $c_u \to \infty$ as $w(u) \to \infty$. Set $c := \inf\{ c_u \}$. By Lemma \ref{L: F estimates}, $F(y) = \sum_{v \in M} B_v y^v$ with $\ord_\gamma(B_v) \ge w(v)$. Consider the product $\xi(y) F(y) = \sum_{r \in M} D_r y^r$ where $D_r := \sum_{u+v=r} C_u B_v$. Note that
\[
\ord_\gamma C_u B_v \ge w(u) + c_u + w(v) \geq w(u+v) + c = w(r) + c.
\]
Thus, $\ord_\gamma D_r \geq w(r) + c$.

Now, 
\[
\alpha(\xi) = \tau^{-1} \circ \psi_p \sum_{r \in M} D_r y^r = \sum_{r \in M} \tau^{-1}(D_{pr}) y^r = \sum_{r \in M} E_r \gamma^{w(r)} y^r,
\]
where $E_r := \tau^{-1}(D_{pr}) \gamma^{-w(r)}$. We are left with checking that $E_r \rightarrow 0$ as $w(r) \rightarrow \infty$.

Since the $p$-adic valuation is invariant under the action of $\tau$, we may ignore it. Thus, 
\[
\ord_\gamma E_r = \ord_\gamma D_{pr} - w(r) \ge w(pr) + c - w(r) = (p-1)w(r) + c.
\]
This proves the proposition.
\end{proof}

Observe that $\alpha_a=\alpha^a$ since
\begin{align*}
\alpha_a &= \psi_q \circ F_a \\
& = \psi_p^a \circ F^{\tau^{a-1}}(y^{p^{a-1}})\cdots F^{\tau}(y^{p}) F(y)\\
&=(\tau^{-1} \circ \psi_p \circ F(y)) \circ \cdots \circ(\tau^{-1} \circ \psi_p \circ F(y)) \\
&=\alpha^a.
\end{align*}
Consequently, $\alpha_a$ is a well-defined completely continuous operator on $B$. 

Next, we linearly extend $\sigma$ to act on elements of $B$.

\begin{lemma}\label{lem:g_commutes_alpha}
As operators on $B$,  $\sigma \circ \alpha_a = \alpha_a \circ \sigma$.
\end{lemma}

\begin{proof}
We will show $\sigma$ commutes with $\psi_q$ and $F_a(y)$, and so it commutes with $\alpha_a$. Let $\xi \in B$. By Lemma \ref{L: F invariant},
\[
\sigma(F_a(y) \cdot \xi(y)) = F_a(\sigma y)\xi(\sigma y) = F_a(y) \sigma( \xi(y)),
\]
thus $\sigma$ commutes with multiplication by $F_a(y)$. Next, we need only check commutativity with $\psi_q$ on the basis monomials. By Lemma \ref{lem:pullback_action}, $\sigma(y^u) = y^{Pu}$. Since $P$ is a permutation matrix, $q \mid u$ if and only if $q \mid Pu$. Thus,
\[
\sigma\big(\psi_q(y^u)\big) = \sigma(y^{u/q}) = y^{P(u/q)} = y^{(Pu)/q} = \psi_q(y^{Pu}) = \psi_q\big(\sigma(y^u)\big).
\]
\end{proof}

\section{The twisted Dwork trace formula}\label{sec:trace_formula}

Let $k \ge 1$, and let $b$ be an integer coprime to $d = \lcm(d_1, \dots, d_n)$. Define the fixed point set
\[
W_k^{(b)} := \big\{ y \in (\overline{\bb F}_p^*)^N \;\big|\; \sigma^{-b}(y^{q^k}) = y \big\}.
\]
Define the map $\rho: W_k^{(b)} \rightarrow \bb F_{q^{k d_1}}^*\times \cdots \times \bb F_{q^{k d_n}}^*$ given by the projection onto the zero-th coordinates $\rho(y) := (y_{1,0}, \dots, y_{n,0})$. The next lemma shows $\rho$ is well-defined and a bijection.

\begin{lemma}\label{lem:domain_bijection}
For $\gcd(b, d) = 1$, the map $\rho$ is well-defined and a bijection. Hence,
\[
\big|W_k^{(b)}\big| = \prod_{i=1}^n (q^{k d_i} - 1) = \det(q^k I - P^b).
\]
\end{lemma}

\begin{proof}
For $y \in W_k^{(b)}$, the coordinate version of $y^{q^k} = \sigma^b(y)$ is the recurrence relation:
\begin{equation} \label{eq:recurrence}
y_{i, j}^{q^k} = y_{i, (j + b) \bmod d_i}.
\end{equation}
Starting at the zero-th coordinate $j=0$ and iterating the recurrence relation, we find $y_{i, b} = y_{i,0}^{q^k}$, then $y_{i, 2b} = y_{i,b}^{q^k} = y_{i,0}^{q^{2k}}$, and so:
\begin{equation}\label{E: recur}
y_{i, mb \bmod d_i} = y_{i,0}^{q^{mk}} \qquad \text{for all } m \ge 0.
\end{equation}
Taking $m = d_i$ yields
\[
y_{i,0} = y_{i,0}^{q^{k d_i}},
\]
and so $y_{i, 0} \in \bb F_{q^{k d_i}}^*$. This proves the map $\rho$ is well-defined. 

Next, let $\bar{z}_i \in \bb F_{q^{k d_i}}^*$ for $i = 1, \ldots, n$. Set $y_{i, 0} := \bar z_i$ for each $i$. Since $\gcd(b, d) = 1$ and $d_i \mid d$, we have $\gcd(b, d_i) = 1$. Thus, as $m$ runs from $0$ to $d_i-1$, the indices $mb \pmod{d_i}$ run over all elements in $\bb Z/d_i \bb Z$  exactly once. Hence, we can use (\ref{E: recur}) to construct the remaining variables: for $m = 1, \ldots, d_i - 1$, set $y_{i, mb \bmod d_i} := y_{i, 0}^{q^{km}}$. Then, $\rho(y) = (\bar z_1, \ldots, \bar z_n)$, showing $\rho$ is surjective. This construction also shows $\rho$ is injective since the map $t \mapsto t^{q^k}$ is an automorphism of $\bb F_{q^{k d_i}}$.

Last, let us show the determinant identity. Observe that the permutation matrix $P$ acts on the variables $y_{i,j}$ independently for each index $i \in \{1, \dots, n\}$. Thus, $P$ decomposes into a block-diagonal matrix, $P = \mathrm{diag}(P_1, \dots, P_n)$, where each $P_i$ is a $d_i \times d_i$ permutation matrix defined by the shift $v_{i, j \bmod d_i} \mapsto v_{i, j-1 \bmod d_i}$, and so $P_i$ has order $d_i$.

Consequently, $P^b$ is also block-diagonal, with $P^b = \mathrm{diag}(P_1^b, \dots, P_n^b)$. Since $b$ is coprime to each $d_i$, every $P_i^b$ has order $d_i$. Hence, the characteristic polynomial of each $P_i^b$ is $t^{d_i} - 1$, and so
\[
\det(tI - P^b) = \prod_{i=1}^n \det(tI_{d_i} - P_i^b) = \prod_{i=1}^n (t^{d_i} - 1).
\]
Evaluating at $t = q^k$ gives:
\[
\det(q^k I - P^b) = \prod_{i=1}^n (q^{k d_i} - 1) = | W_k^{(b)} |.
\]
\end{proof}

The following is a fundamental relation in the classical $p$-adic theory of exponential sum, but here we see it is twisted by $\sigma$ but in the form of summing over $W_k^{(b)}$.

\begin{theorem}\label{T: Sk and F^k}
Let $k \ge 1$ and let $b$ be coprime to $d$. Then, 
\[
S_k(f, \bd) = \sum_{\substack{\bar y \in W_k^{(b)} \\ \hat y = \Teich(\bar y)}} F_a(\hat y) F_a(\hat y^q) \cdots F_a(\hat y^{q^{k-1}}).
\]
\end{theorem}

\begin{proof}
For $y \in W_k^{(b)}$, the recurrence relation \eqref{eq:recurrence} established in Lemma \ref{lem:domain_bijection} gives:
\[
y_{i, mb \bmod d_i} = y_{i,0}^{q^{mk}}
\]
for all integers $m$. Let us evaluate the unfolded polynomial $G(y)$ on this fixed point set. By definition:
\[
G(y) = \sum_{l=0}^{d-1} f\big(y_{1, l\bmod d_1}, \dots, y_{n, l\bmod d_n}\big).
\]
Since $b$ is coprime to $d$, the mapping $m \mapsto mb \pmod d$ is a bijection on $\{0, 1, \dots, d-1\}$. Thus, we can re-index the summation over the shifts $l$ by making the substitution $l = mb \bmod d$:
\begin{align*}
G(y) &= \sum_{m=0}^{d-1} f\big(y_{1, mb \bmod d_1}, \dots, y_{n, mb \bmod d_n}\big) \\
&= \sum_{m=0}^{d-1} f\big(y_{1,0}^{q^{mk}}, \dots, y_{n,0}^{q^{mk}}\big) \\
&= \sum_{m=0}^{d-1} f(y_{1,0}, \dots, y_{n,0})^{q^{mk}} \\
&= \Tr_{\bb F_{q^{kd}} / \bb F_{q^k}} \big( f(y_{1,0}, \dots, y_{n,0}) \big)
\end{align*}
since the coefficients of $f$ lie in $\bb F_q$.

Let $\Theta$ be the additive character on $\bb F_p$ defined by $\Theta(\bar z) := \theta(\hat z)$, where $\theta$ is Dwork's splitting function, and $\hat z = \Teich(\bar z)$. Then by Lemma \ref{lem:domain_bijection},
\begin{align*}
S_k(f, \bd) &:= \sum_{\bar x_1, \dots, \bar x_n \in \bb{F}_{q^{k d_1}}^* \times \cdots \times \bb{F}_{q^{k d_n}}^*} \Theta  \circ \Tr_{\bb{F}_{q^{kd}} / \bb{F}_p} \big( f(\bar x_1, \dots, \bar x_n) \big) \\
&= \sum_{\bar y \in W_k^{(b)}} \Theta  \circ \Tr_{\bb{F}_{q^{k}} / \bb{F}_p} \circ \Tr_{\bb{F}_{q^{kd}} / \bb{F}_{q^k}} \big( f(\bar y_{1,0}, \dots, \bar y_{n,0}) \big) \\
&= \sum_{\bar y \in W_k^{(b)}} \Theta  \circ \Tr_{\bb{F}_{q^{k}} / \bb{F}_p} G(\bar y) \\
&= \sum_{\substack{\bar y \in W_k^{(b)} \\ \hat y = \Teich(\bar y)}} F_a(\hat y) F_a(\hat y^q) \cdots F_a(\hat y^{q^{k-1}}),
\end{align*}
where the last equality comes from the construction of the $F_a$ and the splitting property of $\theta$.
\end{proof}

\begin{lemma}\label{lem:char_orthogonality}
For any $u \in \bb Z^N$, 
\[
\sum_{y \in W_k^{(b)}} y^u 
= 
\begin{cases}
\big|W_k^{(b)}\big| & \text{if } u \in (q^k P^{-b} - I)\bb Z^N, \\
0 & \text{otherwise.}
\end{cases}
\]
\end{lemma}

\begin{proof}
To evaluate the sum $\sum_{y \in W_k^{(b)}} y^u$, we view the map $\chi_u(y) := y^u$ as a character on the finite abelian group $W_k^{(b)}$. By the standard orthogonality relations for characters, the sum over the group is $\big|W_k^{(b)}\big|$ if $\chi_u$ is the trivial character, and $0$ otherwise. Thus, we need only show:
\begin{equation}\label{eq:kernel_claim}
\chi_u \text{ is trivial on } W_k^{(b)} \quad \Longleftrightarrow \quad u \in (q^k P^{-b} - I)\bb Z^N.
\end{equation}

Define the homomorphism
\[
\Phi\colon \bb Z^N \longrightarrow \widehat{W_k^{(b)}},
\qquad
u\longmapsto \chi_u|_{W_k^{(b)}}.
\]
The kernel of $\Phi$ is the set of $u\in\bb Z^N$ such that $\chi_u$ is trivial on $W_k^{(b)}$. Setting $L := (q^k P^{-b} - I)\bb Z^N$, \eqref{eq:kernel_claim} is equivalent to the claim that $\ker(\Phi) = L$.

We first prove that $L\subseteq \ker(\Phi)$. Let $u = (q^k P^{-b} - I)v$ for some $v \in \bb Z^N$, and let $y \in W_k^{(b)}$. By definition of $W_k^{(b)}$, $y^{q^k} = \sigma^b(y)$. Since $(\sigma^b(y))^w = y^{P^b w}$ from Lemma~\ref{lem:pullback_action}, we compute:
\[
\chi_u(y) = y^{(q^k P^{-b} - I)v}
= \big(y^{q^k}\big)^{P^{-b}v} \cdot y^{-v}
= \big(\sigma^b(y)\big)^{P^{-b}v} \cdot y^{-v}
= y^{P^b P^{-b} v} \cdot y^{-v}
= 1.
\]
Hence $\chi_u$ is trivial on $W_k^{(b)}$, which shows $L \subseteq \ker(\Phi)$.

Next, we show $\Phi$ is surjective. Since the torus $T = (\overline{\bb F}_p^\times)^N$ is a divisible abelian group, every character on the subgroup $W_k^{(b)}$ extends to a character on $T$. Since every character of $T$ is of the form $y \mapsto y^u$ for some $u \in \bb Z^N$, every character of $W_k^{(b)}$ lies in the image of $\Phi$. Thus, $\Phi$ is surjective.

Thus, $\bb Z^N/\ker(\Phi) \cong \widehat{W_k^{(b)}}$. Since the character group of a finite abelian group has the same cardinality as the group itself, we have:
\[
\big|\bb Z^N/\ker(\Phi)\big| = \big|\widehat{W_k^{(b)}}\big| = \big|W_k^{(b)}\big|.
\]

Now we compare this with the index of $L$ in $\bb Z^N$. Since $L = (q^k P^{-b} - I)\bb Z^N$, and $P$ is a permutation matrix, its index is given by:
\[
\big|\bb Z^N/L\big| = \big|\det(q^k P^{-b} - I)\big| = \big|\det(q^k I - P^b)\big| =  \big|W_k^{(b)}\big|,
\]
where the last equality comes from Lemma \ref{lem:domain_bijection}. Thus,
\[
\big|\bb Z^N/L\big| = \big|W_k^{(b)}\big| = \big|\bb Z^N/\ker(\Phi)\big|.
\]
Since $L \subseteq \ker(\Phi)$, we must have $L = \ker(\Phi)$. This completes the proof.
\end{proof}

The classical Dwork trace formula for toric exponential sums takes the form \mbox{$S_k^*(f) = (q^k - 1)^n \Tr(\alpha_a^k)$}. The Dwork trace formula for partial $L$-functions has a similar structure:

\begin{theorem}[Twisted Dwork trace formula]\label{thm:twisted_trace}
Let $k \ge 1$ and let $b$ be coprime to $d$. Then, 
\begin{align*}
S_k(f, \bd) &= \left( \prod_{i = 1}^n (q^{k d_i} - 1) \right) \Tr(\sigma^b \circ \alpha_a^k \mid B) \\
&= \det(q^k I - P^b) \,\Tr\!\big(\sigma^b \circ \alpha_a^k \mid B\big).
\end{align*}
\end{theorem}

\begin{proof}
Recall, $\alpha_a = \psi_q \circ F_a(y)$, and so $\alpha_a^k = \psi_{q^k} \circ F_a^{(k)}(y)$, where $F_a^{(k)}(y) := F_a(y) F_a(y^q) \cdots F_a(y^{q^{k-1}})$. Write $F_a^{(k)}(y) = \sum_{v \in M} C_v y^v$. Then, the action of $\alpha_a^k$ on a basis monomial $y^u$ is
\[
\alpha_a^k(y^u) = \psi_{q^k}\big(F_a^{(k)}(y) y^u\big) = \sum_{v \in M} C_{q^k v - u} \, y^v.
\]
Since $\sigma^b(y^v) = y^{P^b v}$, we have
\[
(\sigma^b \circ \alpha_a^k)(y^u) = \sum_{v \in M} C_{q^k v - u} \, y^{P^b v}.
\]
The only terms contributing to the trace $\Tr(\sigma^b \circ \alpha_a^k \mid B)$ are the diagonal entries, which are the terms satisfying $y^{P^b v} = y^u$. Hence, $v = P^{-b} u$. The coefficient for this diagonal term is $C_{q^k P^{-b} u - u}$. Summing over all basis vectors $u \in M$, we obtain the trace:
\[
\Tr\!\big(\sigma^b \circ \alpha_a^k \mid B\big) = \sum_{u \in M} C_{(q^k P^{-b} - I)u}.
\]
On the other hand, applying Lemma \ref{lem:char_orthogonality} we have
\[
\sum_{y \in W_k^{(b)}} F_a^{(k)}(y) = \sum_{y \in W_k^{(b)}} \sum_{u \in M} C_u y^u = \sum_{u \in M} C_u \sum_{y \in W_k^{(b)}} y^u = \big|W_k^{(b)}\big| \sum_{u \in M} C_{(q^k P^{-b} - I)u}.
\]
Hence, by Theorem \ref{T: Sk and F^k} and Lemma \ref{lem:domain_bijection}, we have
\[
S_k(f, \bd) = \sum_{y \in W_k^{(b)}} F_a^{(k)}(y) = \big|W_k^{(b)}\big| \Tr\!\big(\sigma^b \circ \alpha_a^k \mid B\big) = \det(q^k I - P^b) \Tr\!\big(\sigma^b \circ \alpha_a^k \mid B\big).
\]
\end{proof}

\section{The twisted Fredholm determinant}\label{sec:fredholm}

The Twisted Dwork trace formula (Theorem \ref{thm:twisted_trace}) expresses partial toric exponential sums as trace of Dwork's Frobenius twisted by an operator of finite order $d$. To use this in our $L$-functions, we must generalize the classical $p$-adic Fredholm determinant \cite{MR0144186} to accommodate this twist.

Let $V$ be a $p$-adic Banach space over a $p$-adic field $K$. Let $\Phi: V \to V$ be a completely continuous $K$-linear operator, and let $\mathcal{H}: V \to V$ be a bounded operator of finite order $m \ge 1$ that commutes with $\Phi$ (i.e., $\mathcal{H}\Phi = \Phi\mathcal{H}$). 

\begin{definition}\label{def:H_fredholm}
The \emph{twisted Fredholm determinant} of $\Phi$ on $V$ is defined as the formal power series:
\[
\det\nolimits_{\mathcal{H}}(I - T\Phi \mid V) := \exp \left( - \sum_{k=1}^\infty \Tr(\mathcal{H} \circ \Phi^k \mid V) \frac{T^k}{k} \right) \in 1 + T K[[T]].
\]
\end{definition}

Since $\Phi$ is completely continuous and $\mathcal{H}$ is bounded, the composition $\mathcal{H} \circ \Phi^k$ is completely continuous, and so the trace is a well-defined element of $K$. 

Unlike their classical counterparts, twisted Fredholm determinants need not be $p$-adically entire. For example, if $V$ is one-dimensional, $\Phi = \lambda I$, and $\mathcal{H} = \zeta I$ for a nontrivial root of unity $\zeta$, the twisted Fredholm determinant is $(1 - \lambda T)^\zeta$, whose binomial expansion has a finite radius of convergence even though $\Phi$ is trivially completely continuous.

\begin{lemma}\label{lem:nilpotent_trace}
Let $V_\lambda$ be the finite-dimensional generalized $\lambda$-eigenspace of $\Phi$. Then for all $k \ge 1$,
\[
\Tr(\mathcal H \circ \Phi^k \mid V_\lambda) = \lambda^k \Tr(\mathcal H \mid V_\lambda).
\]
\end{lemma}

\begin{proof}
On the finite-dimensional space $V_\lambda$, we may write $\Phi\big|_{V_\lambda} = \lambda I + N$ where $N$ is nilpotent. Since $\mathcal H$ commutes with $\Phi$, $\mathcal H$ commutes with $N$. Hence, $\mathcal H N^j$ is nilpotent for every $j \geq 1$. Now,
\[
\mathcal H \circ \Phi^k = \mathcal H(\lambda I + N)^k = \mathcal H \left( \lambda^k I + \sum_{j=1}^k \binom{k}{j} \lambda^{k-j} N^j \right) = \lambda^k \mathcal H + \sum_{j=1}^k \binom{k}{j} \lambda^{k-j} (\mathcal H N^j).
\]
Since $\mathcal H N^j$ is nilpotent, its trace is zero. Thus, taking the trace of the binomial expansion gives
\[
\Tr(\mathcal H \circ \Phi^k \mid V_\lambda) = \Tr(\lambda^k \mathcal H \mid V_\lambda) = \lambda^k \Tr(\mathcal H \mid V_\lambda).
\]
\end{proof}

By the theory of completely continuous operators, the trace $\Tr(\mathcal H \circ \Phi^k \mid V)$ is the convergent sum of the traces over its generalized eigenspaces. Applying Lemma \ref{lem:nilpotent_trace}, we have $\Tr(\mathcal H \circ \Phi^k \mid V) = \sum_{\lambda \neq 0} c_\lambda \lambda^k$, where $c_\lambda := \Tr(\mathcal H \mid V_\lambda)$ and $| \lambda|_p \rightarrow 0$. Thus, we obtain the following identity within the ring of formal power series $\bb C_p[[T]]$:
\[
\det\nolimits_{\mathcal H}(I - T \Phi \mid V) = \prod_{\lambda \neq 0} (1 - \lambda T)^{c_\lambda}.
\]
Since $\mathcal{H}$ has finite order, its eigenvalues are roots of unity, and so the multiplicities $c_\lambda$ are algebraic integers but not necessarily rational integers $\bb Z$. For a non-integer exponent $c_\lambda$, the $p$-adic binomial expansion of $(1 - \lambda T)^{c_\lambda}$ has a finite radius on convergence. Consequently, the infinite product above is analytic only on a ball of finite radius.

\begin{theorem}
If each multiplicity $c_\lambda \in \bb Z$, then $\det\nolimits_{\mathcal H}(I - T \Phi \mid V)$ is $p$-adic meromorphic.
\end{theorem}

\section{The Partial $\delta_{\bf d}$-Operator}

In classical Dwork theory, the $L$-function of toric exponential sums may be written using Dwork's $\delta$ operator, defined by: $F(T)^\delta := F(T) / F(qT)$. For partial toric sums, we may do the same, however we need to take into account the asymmetric field extensions. 

For a formal power series $F(T) \in 1 + T K[[T]]$ and an integer $c \ge 1$, define $\delta_c$ by:
\[
F(T)^{\delta_c} := \frac{F(T)}{F(q^c T)}.
\]
With $\bd = (d_1, \ldots, d_n)$, we define the partial $\delta_{\bf d}$-operator as the composition:
\[
\delta_{\bd} := \delta_{d_1} \circ \delta_{d_2} \circ \cdots \circ \delta_{d_n}.
\]

\begin{lemma}\label{lem:delta_operator}
Let $F(T) \in 1 + T K[[T]]$. Then
\[
F(T)^{\delta_{\bd}} = \prod_{I \subseteq \{1, \dots, n\}} F(q^{d_I} T)^{(-1)^{|I|}}, \qquad \text{where } d_I := \sum_{i \in I} d_i,
\]
and the empty set $I = \emptyset$ corresponds to $d_\emptyset = 0$, contributing the factor $F(T)^{(-1)^0} = F(T)$.
\end{lemma}

\begin{proof}
We proceed by induction on $n$. For $n=1$, the composition is simply $\delta_{d_1}$, and the formula reads $F(T)^{\delta_{d_1}} = F(q^0 T)^{(-1)^0} F(q^{d_1} T)^{(-1)^1} = F(T) / F(q^{d_1} T)$ as desired.

Assume the identity holds for $n-1$, and write $\bd' = (d_1, \dots, d_{n-1})$. By definition, $F(T)^{\delta_{\bd}} = (F(T)^{\delta_{\bd'}})^{\delta_{d_n}}$. Thus,
\[
F(T)^{\delta_{\bd}} = \frac{F(T)^{\delta_{\bd'}}}{F(q^{d_n} T)^{\delta_{\bd'}}}.
\]
By the induction hypothesis, the numerator expands as $\prod_{I \subseteq \{1, \dots, n-1\}} F(q^{d_I} T)^{(-1)^{|I|}}$, and the denominator expands as $\prod_{I \subseteq \{1, \dots, n-1\}} F(q^{d_I + d_n} T)^{(-1)^{|I|}}$. We move the denominator to the numerator, so it becomes
\[
\prod_{I \subseteq \{1, \dots, n-1\}} F(q^{d_I + d_n} T)^{(-1)^{|I|+1}}.
\]
Set $J := I \cup \{n\} \subseteq \{1, \dots, n\}$. For these subsets, we have $|J| = |I| + 1$ and $d_J = d_I + d_n$. Thus, the exponent $(-1)^{|I| + 1} = (-1)^{|J|}$. The original numerator covers all subsets not containing $n$, and the inverted denominator covers all subsets containing $n$. Combining them gives the product over all $I \subseteq \{1, \dots, n\}$.
\end{proof}

We may now give the $\delta_{\bf d}$ formula for the partial $L$-function. 

\begin{theorem}\label{thm:l_function_factorization}
The partial $L$-function may be written as an alternating product of twisted Fredholm determinants:
\[
L(\bd, f/\bb F_q, T)^{(-1)^{n-1}} = \det\nolimits_\sigma(I - T\alpha_a \mid B)^{\delta_{\bd}}.
\]
\end{theorem}

\begin{proof}
The twisted Dwork trace formula (Theorem \ref{thm:twisted_trace}) with $b = 1$ states
\[
S_k(f, \bd) = \det(q^k I - P) \Tr(\sigma \circ \alpha_a^k \mid B).
\]
By Lemma \ref{lem:domain_bijection}, 
\[
\det(q^k I - P) = \prod_{i=1}^n (q^{k d_i} - 1) = \sum_{I \subseteq \{1, \dots, n\}} (-1)^{n-|I|} q^{k d_I}
\]
where $d_I := \sum_{i \in I} d_i$.  Thus,
\begin{align*}
L(\bd, f/\bb F_q, T) &:= \exp \left( \sum_{k=1}^\infty S_k(f, \bd) \frac{T^k}{k} \right) \\
&= \exp \left( \sum_{k=1}^\infty \sum_{I \subseteq \{1, \dots, n\}} (-1)^{n-|I|} \Tr(\sigma \circ \alpha_a^k \mid B) \frac{(q^{d_I} T)^k}{k} \right) \\
&= \prod_{I \subseteq \{1, \dots, n\}} \exp \left( - \sum_{k=1}^\infty \Tr(\sigma \circ \alpha_a^k \mid B) \frac{(q^{d_I} T)^k}{k} \right)^{(-1)^{n-|I|+1}} \\
&= \prod_{I \subseteq \{1, \dots, n\}} \det\nolimits_\sigma(I - q^{d_I} T \alpha_a \mid B)^{(-1)^{n-|I|+1}}.
\end{align*}
Hence, by Lemma \ref{lem:delta_operator}, we have
\begin{align*}
L(\bd, f/\bb F_q, T)^{(-1)^{n-1}} &= \prod_{I \subseteq \{1, \dots, n\}} \det\nolimits_\sigma(I - q^{d_I} T \alpha_a \mid B)^{(-1)^{|I|}} \\
&= \det\nolimits_\sigma(I - T\alpha_a \mid B)^{\delta_{\bd}}
\end{align*}
as desired.
\end{proof}

\section{Rationality of the partial $L$-function}

For an integer $b$ coprime to $d$, consider the twisted Fredholm determinant on the Banach space $B$:
\[
\det\nolimits_{\sigma^b}(I - T\alpha_a \mid B) := \exp \left( - \sum_{k=1}^\infty \Tr(\sigma^b \circ \alpha_a^k \mid B) \frac{T^k}{k} \right).
\]
As discussed in Section \ref{sec:fredholm}, we may write 
\[
\det\nolimits_{\sigma^b}(I - T\alpha_a \mid B) = \prod_{\lambda \neq 0} (1 - \lambda T)^{c_\lambda(b)}, \qquad \text{where } c_\lambda(b) := \Tr(\sigma^b \mid V_\lambda),
\]
and $V_\lambda$ denotes the generalized $\lambda$-eigenspace of $\alpha_a$.

\begin{theorem}\label{thm:integrality}
$c_\lambda(1) \in \mathbb{Z}$ for every eigenvalue $\lambda$.
\end{theorem}

\begin{proof}
We will first show that the traces $c_\lambda(b)$ are independent of the choice of $b$. We start by recalling the twisted Dwork trace formula (Theorem \ref{thm:twisted_trace}):
$$S_k(f, \mathbf{d}) = \det(q^k I - P^b) \operatorname{Tr}(\sigma^b \circ \alpha_a^k \mid B).$$

By Lemma \ref{lem:domain_bijection}, the determinant $\det(q^k I - P^b) = \det(q^k I - P)$ is independent of $b$. Since the partial exponential sum $S_k(f, \mathbf{d})$ is independent of $b$, we get that $\operatorname{Tr}(\sigma^b \circ \alpha_a^k \mid B) = \operatorname{Tr}(\sigma \circ \alpha_a^k \mid B)$ for all $b$ coprime to $d$, and for all integers $k \ge 1$.

Next, since $\alpha_a$ is a completely continuous operator, we may expand the trace of its $k$-th iterate as the convergent sum over its generalized $\lambda$-eigenspaces $V_\lambda$:
$$
\operatorname{Tr}(\sigma^b \circ \alpha_a^k \mid B) = \sum_{\lambda \neq 0} c_\lambda(b) \lambda^k = \sum_{\lambda \neq 0} c_\lambda(1) \lambda^k,
$$
where $c_\lambda(b) := \operatorname{Tr}(\sigma^b \mid V_\lambda)$ and $c_\lambda(1) := \operatorname{Tr}(\sigma \mid V_\lambda)$. Defining the difference $d_\lambda := c_\lambda(b) - c_\lambda(1)$, we obtain the sequence of identities $\sum_{\lambda \neq 0} d_\lambda \lambda^k = 0$ for all integers $k \ge 1$. We will show $d_\lambda = 0$ for every $\lambda$. 

Define the formal generating function:
$$\sum_{k=1}^\infty \left( \sum_{\lambda \neq 0} d_\lambda \lambda^k \right) T^{k-1} = 0.$$
Since the eigenvalues of a completely continuous operator satisfy $|\lambda|_p \to 0$, we may change the order of summation to get
$$\sum_{\lambda \neq 0} \frac{d_\lambda \lambda}{1 - \lambda T} = 0.$$

Fix an eigenvalue $\lambda_ 0$. Write
$$\frac{d_{\lambda_0} \lambda_0}{1 - \lambda_0 T} + \sum_{\lambda \neq 0, \lambda_0} \frac{d_\lambda \lambda}{1 - \lambda T} = 0$$
as
$$d_{\lambda_0} \lambda_0 + (1 - \lambda_0 T) \sum_{\lambda \neq 0, \lambda_0} \frac{d_\lambda \lambda}{1 - \lambda T} = 0.$$
Since the lefthand side is $p$-adic meromorphic, we may evaluate it at $T = \lambda_0^{-1}$ to obtain $d_{\lambda_0} \lambda_0 = 0$. Thus, $d_{\lambda_0} = 0$. Consequently, $c_\lambda(b) = c_\lambda(1)$ for all $b \in (\mathbb{Z}/d\mathbb{Z})^\times$.

Next, since the shift operator $\sigma$ has order $d$, its restriction to $V_\lambda$ has order dividing $d$. Its eigenvalues are therefore $d$-th roots of unity, which implies that the trace $c_\lambda(1) \in \mathbb{Q}(\zeta_d)$.

The Galois group $\operatorname{Gal}(\mathbb{Q}(\zeta_d) / \mathbb{Q})$ consists of all automorphisms $\tau_b : \zeta \mapsto \zeta^b$ for $b \in (\mathbb{Z}/d\mathbb{Z})^\times$. Applying $\tau_b$ to the trace of $\sigma$ raises its roots of unity to the $b$-th power. Thus, 
$$\tau_b(c_\lambda(1)) = \operatorname{Tr}(\sigma^b \mid V_\lambda) = c_\lambda(b) = c_\lambda(1).$$

Since $c_\lambda(1)$ is fixed by the Galois group, it must lie in the base field $\mathbb{Q}$. Since it is also a finite sum of algebraic integers (roots of unity), $c_\lambda(1)$ is an algebraic integer. Therefore, $c_\lambda(1) \in \mathbb{Z}$, as desired.
\end{proof}

\begin{corollary}
The twisted Fredholm determinant $\det\nolimits_{\sigma}(I - T\alpha_a \mid B)$ is $p$-adic meromorphic. Consequently, the partial $L$-function $L(\bd, f/\bb{F}_q, T)$ is $p$-adic meromorphic, and hence, is a rational function in $\bb Q(\zeta_p)(T)$.
\end{corollary}

\begin{proof}
Since the traces $c_\lambda(1)$ are integers (Theorem \ref{thm:integrality}), we see from Section \ref{sec:fredholm} that \mbox{$\det\nolimits_{\sigma}(I - T\alpha_a \mid B)$} is $p$-adic meromorphic. Next, the $\delta_{\bf d}$ structure of the partial $L$-function (Theorem \ref{thm:l_function_factorization}) shows that the $L$-function is $p$-adic meromorphic. We are now able to use Dwork's standard rationality criterion: any formal power series with coefficients in a finite degree algebraic number field that is $p$-adically meromorphic and complex analytic in a small disk must be a rational function.
\end{proof}

\section{$p$-adic Cohomology} \label{S: cohom}

In this section we develop a $p$-adic cohomology theory for the partial $L$-function using a Koszul complex adapted to fit the twisted framework above. To ease reading, we recall some definitions from Section \ref{S: Banach}. 

Let $\bb Q_q$ be the unramified extension of $\bb Q_p$ of degree $a$, and let $\bb Z_q$ be its ring of integers. By the theory of Newton polygons, the series $\sum_{i \ge 0} t^{p^i}/p^i$ has a root $\gamma \in \overline{\bb Q}_p$ satisfying $\ord_p \gamma = 1/(p-1)$. Fix a $D$-th root $\gamma^{1/D}$, and set $K := \bb Q_q(\gamma^{1/D})$. Define the $p$-adic Banach space $B$ over $K$ with orthonormal basis $\{ \gamma^{w(u)} y^u \}_{u \in M}$ as:
\[
B := \left\{ \sum_{u \in M} A_u \gamma^{w(u)} y^u \;\middle|\; A_u \in K \text{ and }  |A_u|_p \rightarrow 0 \text{ as } w(u) \to \infty \right\}.
\]
Write $f(x) = \sum_{u \in \supp(f)} \bar c_u x^u$ with coefficients $\bar c_u \in \bb F_q^*$, and let $c_u = \Teich(\bar c_u) \in \bb Z_q$ be their Teichm\"uller lifts. The unfolded polynomial associated to $f$ is defined by
\[
G(y) := \sum_{l=0}^{d-1} f\big(y_{1, l \bmod d_1}, \dots, y_{n, l \bmod d_n}\big) \in \bb F_q[y_{i,j}^{\pm 1} \mid 1 \le i \le n, j \in \bb Z/ d_i \bb Z].
\]
Define
\[
F(y) := \prod_{l=0}^{d-1} \prod_{u \in \supp(f)} \theta\!\left( c_u \prod_{i=1}^n y_{i, l \bmod d_i}^{u_i} \right).
\]
Let $\tau \in \Gal(\bb Q_q / \bb Q_p)$ be the generator of the cyclic Galois group which acts on Teichm\"ullers by $\tau(c) = c^p$. We defined
\[
F_a(y) := F(y) F^\tau(y^p) \cdots F^{\tau^{a-1}}(y^{p^{a-1}}).
\]
Define the Frobenius maps $\alpha := \tau^{-1} \circ \psi_p \circ F(y)$, and $\alpha_a := \alpha^a = \psi_q \circ F_a(y)$.

To define a differential that commutes with the Frobenius, we define the infinite product
\[
K(y) := \prod_{j=0}^\infty F^{\tau^j}(y^{p^j}),
\]
and observe that $F(y) = \frac{K(y)}{K^\tau(y^p)}$. Let $\widehat{G}(y) := \sum_{l=0}^{d-1} \sum_{u \in \supp(f)} c_u \prod_{i=1}^n y_{i, l \bmod d_i}^{u_i}$ be the Teichm\"uller lift of the unfolded polynomial $G(y)$, and denote by $\widehat{G}^{\tau^m}(y)$ the polynomial obtained by applying $\tau^m$ to its coefficients. Now,
\begin{align*}
\log K(y) &= \sum_{j=0}^\infty \log F^{\tau^j}(y^{p^j}) \\
&= \sum_{j=0}^\infty \sum_{i=0}^\infty \frac{\gamma^{p^i}}{p^i} \widehat{G}^{\tau^{i+j}}(y^{p^{i+j}}).
\end{align*}
Collecting terms by setting $m = i+j$, we define $H(y) := \log K(y)$ as
\[
H(y) = \sum_{m=0}^\infty \gamma_m \widehat{G}^{\tau^m}(y^{p^m}), \qquad \text{where} \quad \gamma_m := \sum_{i=0}^m \frac{\gamma^{p^i}}{p^i}.
\]
Notice that $\gamma_m = -\sum_{j = m+1}^\infty \gamma^{p^j}/p^j$, which shows $\ord_p(\gamma_m) = \frac{p^{m+1}}{p-1} - (m+1)$.

For each variable $y_{i, j}$ where $1 \le i \le n$ and $j \in \bb Z / d_i \bb Z$, we define the differential operator $D_{i,j}$ on the Banach space $B$ via conjugation by $K(y)$:
\begin{align*}
D_{i,j} :&= \frac{1}{K(y)} \circ y_{i,j} \frac{\partial}{\partial y_{i,j}} \circ K(y) \\
&= y_{i,j} \frac{\partial}{\partial y_{i,j}} + y_{i,j} \frac{\partial H(y)}{\partial y_{i,j}}.
\end{align*}

\begin{lemma} \label{lem:differential_well_defined}
$y_{i,j} \frac{\partial H(y)}{\partial y_{i,j}} \in B$. Consequently, the differential $D_{i,j}$ is a well-defined endomorphism of $B$.
\end{lemma}

\begin{proof}
The result follows from
\[
y_{i,j} \frac{\partial H(y)}{\partial y_{i,j}} = \sum_{m=0}^\infty \gamma_m p^m y_{i,j} \left( \frac{\partial \widehat{G}^{\tau^m}}{\partial y_{i,j}} \right) (y^{p^m})
\]
and the $p$-adic estimate of $\gamma_m$.
\end{proof}

We may now construct the Koszul complex. Let $N = \sum_{i=1}^n d_i$, and let $W = \mathrm{span}_K\!\left(\frac{dy_{1,0}}{y_{1,0}}, \dots, \frac{dy_{n, d_n-1}}{y_{n, d_n-1}}\right)$ be the $N$-dimensional $K$-vector space of logarithmic differential forms. Define $\Omega^0(B, D) := B$, and for $m = 1, \ldots, N$, define
\[
\Omega^m(B, D) := B \otimes_K \wedge^m W
\]
with boundary map $D: \Omega^m(B, D) \to \Omega^{m+1}(B, D)$ defined by 
\[
D\left( \xi \frac{dy_{i_1, j_1}}{y_{i_1, j_1}} \wedge \cdots \wedge \frac{dy_{i_m, j_m}}{y_{i_m, j_m}} \right) := \left( \sum_{i=1}^n \sum_{j=0}^{d_i-1} D_{i,j}(\xi) \frac{dy_{i,j}}{y_{i,j}} \right) \wedge \frac{dy_{i_1, j_1}}{y_{i_1, j_1}} \wedge \cdots \wedge \frac{dy_{i_m, j_m}}{y_{i_m, j_m}}.
\]
One may check that $D^2 = 0$. Thus, we may define the cohomology spaces
\[
H^m\!\big(B, D\big) := \frac{\ker(D: \Omega^m(B) \to \Omega^{m+1}(B))}{\mathrm{im}(D: \Omega^{m-1}(B) \to \Omega^m(B))}.
\]

Next we define a chain map from Dwork's Frobenius $\alpha_a$. Using the relation $F(y) = \frac{K(y)}{K^\tau(y^p)}$, we may write 
\begin{align*}
\alpha &= \tau^{-1} \circ \psi_p \circ \left( \frac{K(y)}{K^\tau(y^p)} \right) \\
&= \tau^{-1} \circ \left( \frac{1}{K^\tau(y)} \circ \psi_p \circ K(y) \right) \\
&= \frac{1}{K(y)} \circ \tau^{-1} \circ \psi_p \circ K(y).
\end{align*}
Now, since $p y_{i,j} \frac{\partial}{\partial y_{i,j}} \circ \psi_p = \psi_p \circ y_{i,j} \frac{\partial}{\partial y_{i,j}}$ we have the following relation.

\begin{lemma} \label{lem:D_commutes_alpha}
$p D_{i,j} \circ \alpha = \alpha \circ D_{i,j}$. Consequently, $q D_{i,j} \circ \alpha_a = \alpha_a \circ D_{i,j}$.
\end{lemma}

Define the Frobenius chain map $\Frob_q^\bullet : \Omega^\bullet(B, D) \to \Omega^\bullet(B, D)$ by:
\[
\Frob_q^m := q^{N-m} \alpha_a \otimes \mathrm{id}_{\wedge^m W}.
\]

The shift operator $\sigma$ acts on $B$ by $\sigma(y_{i, j \bmod d_i}) = y_{i, (j+1) \bmod d_i}$. Since the unfolded polynomial $G(y)$ is invariant under $\sigma$ (Lemma \ref{lem:G_shift_invariant}), its lifted polynomial $\widehat G(y)$ and the formal series $H(y)$ are also invariant. Thus, we have $\sigma \circ D_{i, j \bmod d_i} = D_{i, (j+1) \bmod d_i} \circ \sigma$. 

Next, $\sigma$ acts on forms 
\[
\sigma\left( \frac{dy_{i,j}}{y_{i,j}} \right) = \frac{d(\sigma y_{i,j})}{\sigma y_{i,j}} = \frac{dy_{i, j+1 \bmod d_i}}{y_{i, j+1 \bmod d_i}}.
\]
Define the chain map $\mathcal{H}^\bullet : \Omega^\bullet(B, D) \to \Omega^\bullet(B, D)$ by
\[
\mathcal{H}^{(m)} := \sigma \otimes \wedge^m \sigma.
\]

\begin{lemma}
$D \circ \mathcal{H}^{(m)} = \mathcal{H}^{(m+1)} \circ D$.
\end{lemma}

\begin{proof}
Let $\omega = \xi \otimes \eta \in \Omega^m(B, D)$, with $\xi \in B$ and $\eta \in \wedge^m W$. Then
\[
D(\xi \otimes \eta) = \sum_{i=1}^n \sum_{j=0}^{d_i-1} D_{i,j}(\xi) \otimes (w_{i,j} \wedge \eta),
\]
where $w_{i,j} := \frac{dy_{i,j}}{y_{i,j}}$ to ease notation. Applying the operator $\mathcal{H}^{(m+1)}$ we see that
\begin{align*}
\mathcal{H}^{(m+1)}\big(D(\xi \otimes \eta)\big) &= \sum_{i=1}^n \sum_{j=0}^{d_i-1} \sigma\big(D_{i,j}(\xi)\big) \otimes \big(\sigma(w_{i,j}) \wedge \sigma(\eta)\big) \\
&= \sum_{i=1}^n \sum_{j=0}^{d_i-1} D_{i, (j+1) \bmod d_i}\big(\sigma(\xi)\big) \otimes \big(w_{i, (j+1) \bmod d_i} \wedge \sigma(\eta)\big) \\
&= \sum_{i=1}^n \sum_{k=0}^{d_i-1} D_{i, k}\big(\sigma(\xi)\big) \otimes \big(w_{i, k} \wedge \sigma(\eta)\big) \\ 
&= D\big(\sigma(\xi) \otimes \sigma(\eta)\big) \\
&= D\big(\mathcal{H}^{(m)}(\omega)\big)
\end{align*}
as desired.
\end{proof}

Since $\sigma$ and $\alpha_a$ commute (Lemma \ref{lem:g_commutes_alpha}), we have the following.

\begin{lemma}
$\Frob_q^m \circ \mathcal{H}^{(m)} = \mathcal{H}^{(m)} \circ \Frob_q^m$
\end{lemma}

As a consequence, $\mathcal{H}^{(m)} \circ \Frob_q^m$ is a well-defined endomorphism on $H^m(B, D)$.

\begin{theorem}\label{thm:koszul_trace}
For each integer $k \ge 1$,
\[
S_k(f, \bd) = \sum_{m=0}^N (-1)^m \Tr\!\left( \mathcal{H}^{(m)} \circ (\Frob_q^m)^k \ \Big|\ \Omega^m(B)\right).
\]
\end{theorem}

\begin{proof}
Since $\Omega^m(B, D) = B \otimes_K \wedge^m W$ is a tensor product and $(\Frob_q^m)^k = q^{k(N-m)}\alpha_a^k \otimes \mathrm{id}$ and $\mathcal{H}^{(m)} = \sigma \otimes \wedge^m \sigma$, the trace over the space decomposes multiplicatively:
\[
\Tr\!\left( \mathcal{H}^{(m)} \circ (\Frob_q^m)^k \mid \Omega^m(B, D)\right) = q^{k(N-m)} \Tr(\sigma \circ \alpha_a^k \mid B) \cdot \Tr(\wedge^m \sigma).
\]
Thus,
\begin{align*}
\sum_{m=0}^N (-1)^m \Tr\!\left( \mathcal{H}^{(m)} \circ (\Frob_q^m)^k \ \Big|\ \Omega^m(B)\right) &= \Tr(\sigma \circ \alpha_a^k \mid B) \sum_{m=0}^N (-1)^m (q^k)^{N-m} \Tr(\wedge^m \sigma) \\
&=  \Tr(\sigma \circ \alpha_a^k \mid B) \det(q^k I - \sigma)
\end{align*}
since the characteristic polynomial satisfies the relation $\sum_{m=0}^N (-1)^m t^{N-m} \Tr(\wedge^m A) = \det(tI - A)$. Since $\det(q^k I - \sigma) = \det(q^k I - P)$, the result follows by the twisted Dwork trace formula  (Theorem \ref{thm:twisted_trace}). 
\end{proof}

\begin{corollary}\label{thm:l_function_determinant_cohomology}
\[
L(\bd, f/\bb F_q, T) = \prod_{m=0}^N \det\nolimits_{\mathcal{H}^{(m)}} \! \left(I - T\Frob_q^m \ \Big|\ H^m(B, D)\right)^{(-1)^{m+1}}.
\]
\end{corollary}

In classical Dwork theory, if a Laurent polynomial $f$ is non-degenerate with respect to its Newton polyhedron $\Delta(f)$ and the dimension of $\Delta(f)$ is full, then the associated Koszul complex is acyclic except in the top degree. One might hope that if $f$ is classically non-degenerate, its unfolded polynomial $G(y) = \sum_{l=0}^{d-1} f(y_{(l)})$ is also automatically non-degenerate. Unfortunately, this is not necessarily the case as the following example shows.

For $p \neq 2$, $f(x_1, x_2) = x_1 x_2 + x_1 x_2^{-1}$ is classically non-degenerate. With $d_1 = 1$ and $d_2 = 2$, its unfolded polynomial is
\[
G(y) = f(y_{1,0}, y_{2,0}) + f(y_{1,0}, y_{2,1}) = y_{1,0}(y_{2,0} + y_{2,0}^{-1} + y_{2,1} + y_{2,1}^{-1}).
\]
Since the variable $y_{1,0}$ is shared among the partial derivatives of $G$, we see that $G$ is degenerate with respect to $\Delta(G)$. Thus, we define non-degeneracy of partial $L$-functions on the level of the unfolded polynomial $G$.

\begin{definition} \label{def:d_non_degenerate}
A Laurent polynomial $f \in \bb F_q[x_1^{\pm 1}, \dots, x_n^{\pm 1}]$ is said to be \emph{$\bd$-non-degenerate} if its unfolded polynomial $G(y) = \sum_{l=0}^{d-1} f(y_{(l)})$ is classically non-degenerate with respect to its Newton polyhedron $\Delta(G)$. That is, for every face $\tau \subseteq \Delta(G)$ not containing the origin, the system of partial derivatives $y_{i,j} \frac{\partial G_\tau}{\partial y_{i,j}} = 0$ has no common solution in the algebraic torus $(\overline{\bb F}_q^*)^N$.
\end{definition}

Since the cohomology $H^\bullet(B, D)$ developed in Section \ref{S: cohom} is the usual Koszul cohomology associated to the polynomial $G$, we may apply the theory of Adolphson and Sperber \cite{AdolpSperb-ExponentialSumsand-1989} directly. Assume $f$ is $\bd$-non-degenerate, then $G$ satisfies the classical non-degeneracy conditions. Assuming the dimension of $\Delta(G)$ is $N$, the Koszul complex is acyclic: $H^m(B, D) = 0$ for all $m \neq N$, and $H^N(B, D)$ is a finite-dimensional $K$-vector space of dimension $N! \> \text{Vol}(\Delta(G))$. A precise statement may also be made in the case when $f$ is ${\bf d}$-non-degenerate and $\dim \Delta(G) < N$; see \cite{AdolpSperb-ExponentialSumsand-1989} for details.

Consequently, the partial $L$-function becomes
\begin{align*}
L(\bd, f/\bb F_q, T)^{(-1)^{N+1}} &= \det\nolimits_{\mathcal{H}^{(N)}} \! \left(I - T\Frob_q^N \ \Big|\ H^N(B, D)\right) \\
&= \det\nolimits_{\operatorname{sgn}(P)\sigma} \! \Big( I-T\alpha_a \, \Big| \, B / \sum_{i,j }D_{i,j}B \Big ).
\end{align*}
We emphasize that this determinant is a \emph{twisted} Fredholm determinant in the sense of Section~\ref{sec:fredholm}. Thus, at this point, we only know that it is a rational function even in this non-degenerate situation. For this reason, although it is natural to seek a Newton-over-Hodge statement analogous to the classical theory, at the moment it is not clear what such a statement would be, even when non-degenerate.

\section{Unique $p$-adic unit root}\label{sec:unique_root}

In this section we prove that the partial \(L\)-function $L(\bd,f/\F_q,T)^{(-1)^{n-1}}$ has a unique \(p\)-adic unit root, and that this root is a \(1\)-unit. In Section \ref{sec:unit_root_formula} we give a formula for the unit root in terms of $\mathcal{A}$-hypergeometric functions.

Recall from Section \ref{S: Banach} that the $p$-adic Banach space $B$ has orthonormal basis $e_u := \gamma^{w(u)} y^u$ for $u \in M$. In \cite{MR2966711}, Adolphson and Sperber proved that the (classical) $p$-adic Fredholm determinant $\det(I - T\alpha_a \mid B)$ has exactly one reciprocal root $\lambda_0$ that is a $p$-adic unit, and further, $\lambda_0$ is a $1$-unit, meaning $| \lambda_0 - 1 |_p < 1$.

Let $V_{\lambda_0} \subset B$ be the 1-dimensional eigenspace associated to $\lambda_0$, spanned by the eigenvector $\xi_0 = \sum_{u \in M} C_u e_u \in B$. By Lemma \ref{lem:g_commutes_alpha}, $\sigma$ commutes with $\alpha_a$, and thus $\sigma$ preserves the eigenspaces of $\alpha_a$. Since $V_{\lambda_0}$ is $1$-dimensional, $\sigma$ maps $\xi_0$ to a scalar multiple of itself: $\sigma(\xi_0) = \zeta \xi_0$ for some root of unity $\zeta$. We claim that $\zeta = 1$.

\begin{lemma}\label{lem:twined_trace_unit}
$\zeta = \Tr(\sigma \mid V_{\lambda_0}) = 1$.
\end{lemma}

\begin{proof}
In \cite{MR2966711}, Adolphson and Sperber show that the matrix $(A_{u,v})_{u,v \in M}$ of $\alpha_a$ with respect to the basis $\{e_u\}_{u \in M}$ satisfies
\[
A_{0,0} \equiv 1 \pmod{\gamma}, \qquad A_{u,v} \equiv 0 \pmod{\gamma} \ \text{ if } \ u, v \neq 0.
\]
Let us first show $C_0 \neq 0$. Suppose for contradiction that $C_0 = 0$, and rescale so that $\|\xi_0\| = \sup_{u \neq 0} |C_u|_p = 1$. For any $u \neq 0$, it follows from $\alpha_a(\xi_0) = \lambda_0 \xi_0$ that
\[
\lambda_0 C_u = \sum_{v \in M} A_{u,v} C_v = \sum_{v \neq 0} A_{u,v} C_v,
\]
where the $v = 0$ term vanishes since $C_0 = 0$. Since $|A_{u,v}|_p \leq |\gamma|_p$ for $u, v \neq 0$ and $|\lambda_0|_p = 1$, we have 
\[
|C_u|_p \leq \max_{v \neq 0} |A_{u,v}|_p \, |C_v|_p \leq |\gamma|_p < 1.
\]
Taking the supremum over $u \neq 0$ contradicts $\sup_{u \neq 0} |C_u|_p = 1$. Hence $C_0 \neq 0$.

Next, by Lemma \ref{lem:weight_invariant}, $w(Pu) = w(u)$, and so $\sigma$ acts on the basis by $\sigma(e_u) = \gamma^{w(u)} y^{Pu} = e_{Pu}$. In particular, $\sigma(e_0) = e_0$, and thus
\[
\sigma(\xi_0) = C_0 e_0 + \sum_{\substack{u \in M \\ u \neq 0}} C_u e_{Pu}.
\]
Comparing the $e_0$-coefficient of $\sigma(\xi_0) = \zeta \xi_0$ gives $C_0 \zeta = C_0$, and since $C_0 \neq 0$ we have $\zeta = 1$ as desired.
\end{proof}

\begin{theorem}\label{thm:unique_unit_root_existence}
The partial $L$-function $L(\bd, f/\bb F_q, T)^{(-1)^{n-1}}$ has exactly one reciprocal $p$-adic unit root, which is the unique unit root $\lambda_0$ of $\det(I - T\alpha_a \mid B)$.
\end{theorem}

\begin{proof}
By Theorem \ref{thm:l_function_factorization},
\[
L(\bd, f/\bb F_q, T)^{(-1)^{n-1}} = \det\nolimits_\sigma(I - T\alpha_a \mid B) \cdot \prod_{\substack{I \subseteq \{1, \dots, n\} \\ I \not= \emptyset}} \det\nolimits_\sigma \big(I - q^{d_I} T \alpha_a \mid B\big)^{(-1)^{|I|}}.
\]

By classical Dwork theory, all of the eigenvalues $\lambda$ of $\alpha_a$ on $B$ satisfy $\ord_p(\lambda) \ge 0$. From the formal eigenspace factorization of the twisted Fredholm determinant established in Section \ref{sec:fredholm}, we can expand the twisted Fredholm determinant over the eigenspaces $V_\lambda$ of $\alpha_a$:
\[
\det\nolimits_\sigma \big(I - T \alpha_a \mid B\big) = \prod_{\lambda \neq 0} (1 - \lambda T)^{\Tr(\sigma \mid V_\lambda)}.
\]
By Lemma \ref{thm:integrality}, $\Tr(\sigma \mid V_\lambda) \in \bb Z$, and so due to the $q^{d_I}$, the only term in the product that could have a unit root is from $\det\nolimits_\sigma(I - T\alpha_a \mid B)$. By \cite{MR2966711}, $\alpha_a$ has exactly one unit eigenvalue $\lambda_0$, which is a $1$-unit, and by Lemma \ref{lem:twined_trace_unit}, the trace multiplicity of $\sigma$ on the corresponding unit eigenspace is $\Tr(\sigma \mid V_{\lambda_0}) = 1$.
\end{proof}

\section{Unit root formula}\label{sec:unit_root_formula}

In Section \ref{sec:unique_root}, we showed that the unique unit root the partial $L$-function is the same as the unit root in Adolphson and Sperber's result \cite{MR2966711}. Thus, their unit root formula applies here. For completeness, we describe it here and compare this formula with the classical formula.

Let $f(x) = \sum_{u \in \supp(f)} \bar c_u x^u \in \bb F_q[x_1^{\pm 1}, \dots, x_n^{\pm 1}]$. We introduce parameters $\Lambda = (\Lambda_u)_{u \in \supp f}$ for each monomial in $f$, and define $f_\Lambda(x) := \sum_{u \in \supp(f)} \Lambda_u x^u$. Write
\[
\exp\left( \gamma f_\Lambda(x) \right) = \sum_{v \in \bb Z^n} A_v(\Lambda) x^v,
\]
where the coefficients may be explicitly written as:
\[
A_v(\Lambda) = \sum \prod_{u \in \supp(f)} \frac{(\gamma \Lambda_u)^{k_u}}{k_u!},
\]
where the sum runs over all $k_u \geq 0$ for $u \in \supp(f)$ such that $\sum k_u u = v$. Next, by definition of the unfolded polynomial $G(y)$, we have
\[
\exp\big(\gamma G_\Lambda(y)\big) = \prod_{l=0}^{d-1} \exp\left(\gamma  f_\Lambda(y_{(l)}) \right) = \sum_{v \in \bb Z^N} G_v(\Lambda) y^v,
\]
where $y_{(l)} := (y_{1, l \bmod d_1}, \dots, y_{n, l \bmod d_n})$. In particular, 
\[
G_0(\Lambda) = \sum_{\mathbf{w} \in \mathcal{W}_{\bd}} \prod_{l=0}^{d-1} A_{w^{(l)}}(\Lambda),
\]
where $\mathcal{W}_{\bd}$ is the set of all $d$-tuples of vectors $\mathbf{w} = (w^{(0)}, \dots, w^{(d-1)}) \in (\bb Z^n)^d$ satisfying: for every $1 \le i \le n$ and $j \in \bb Z / d_i \bb Z$,
\begin{equation}\label{E: balance}
\sum_{m=0}^{(d/d_i) - 1} w_i^{(j + m d_i)} = 0.
\end{equation}
Adolphson and Sperber \cite{MR2966711} proved that 
\[
\mathcal{F}(\Lambda) := \frac{G_0(\Lambda)}{G_0(\Lambda^p)}
\]
is $p$-adic convergent on the closed unit polydisk in $\Lambda$. Specializing this to the coefficients of $f$ gives the unit root:

\begin{theorem}\label{thm:main_unit_root_formula}
For $f(x) = \sum \bar c_u x^u \in \bb F_q[x_1^\pm, \ldots, x_n^\pm]$, let $\hat c_u$ denote the Teichm\"uller lift of $\bar c_u$. The unique unit root $\lambda_0$ of the partial $L$-function $L(\bd, f/\bb F_q, T)^{(-1)^{n-1}}$ is given by specializing $\Lambda_u = \hat c_u$:
\[
\lambda_0 = \mathcal{F}(\hat c) \mathcal{F}(\hat c^p) \cdots \mathcal{F}(\hat c^{p^{a - 1}}).
\]
\end{theorem}

We conclude this section by comparing the unit root formula with the classical case $d_1 = \dots = d_n = d$ and the asymmetric case $d_1 = 1$, $d_2 = 2$, and $d=2$.

Suppose $d_1 = \dots = d_n = d$. Then in equation (\ref{E: balance}), since $d/d_i = 1$ we have $w_i^{(j)} = 0$ for all $i$ and $j$. This forces every $w^{(l)} = 0$ for every $l$. Thus,
\[
G_0(\Lambda) = \prod_{l=0}^{d-1} A_0(\Lambda) = \big( A_0(\Lambda) \big)^d.
\]
as expected. 

Now for the asymmetric case. Suppose $n=2$, and $d_1 = 1$, $d_2 = 2$, and $d=2$. In this case there are two vectors $w^{(0)}, w^{(1)} \in \bb Z^2$, and equation (\ref{E: balance}) becomes: 
\[
\text{for $i=1$: $w_1^{(1)} = -w_1^{(0)}$} \qquad \text{and} \qquad \text{for $i=2$: $w_2^{(0)} = 0$ and $w_2^{(1)} = 0$.}
\]
Letting $k = w_1^{(0)}$, the valid tuples are exactly $w^{(0)} = (k, 0)$ and $w^{(1)} = (-k, 0)$. Then
\[
G_0(\Lambda) = \sum_{k \in \bb Z} A_{(k, 0)}(\Lambda) A_{(-k, 0)}(\Lambda).
\]

\bibliographystyle{amsplain}
\bibliography{../References/References}

\noindent
\texttt{haessig@math.arizona.edu}

\end{document}